\documentclass[a4paper,twoside,12pt]{article}      
\usepackage[french]{babel}
\usepackage{amsmath,amssymb,amsfonts} 
\usepackage{amsthm}
\usepackage{graphics}                 
\usepackage{color}                    
\usepackage{hyperref}                 
\usepackage{graphicx,eucal}
\usepackage{latexsym,epsfig,epic,eepic}

\newtheorem{theorem}{Theorem}[section]
\newtheorem{proposition}[theorem]{Proposition}

\theoremstyle{definition}

\date{D\'ecembre 2010}
\title{Feuilletage de Hirsch, mesures harmoniques et $g$-mesures}

\author{Bertrand Deroin et Constantin Vernicos}

\begin{document}
\maketitle

\begin{small}

\section{Introduction}

Un feuilletage lisse $\mathcal F$ est la donn\'ee d'une vari\'et\'e $M$ et d'un atlas $\{  \Phi_i \}_{i\in I}$ sur $M$ form\'e de diff\'eomor\-phismes $\Phi_i : U_i \rightarrow {\bf R} ^ p \times {\bf R} ^q$ tels que les changements de cartes $\Phi_j\circ \Phi_i^{-1}$ s'expriment par~:
\[  \Phi _j \circ \Phi_i ^{-1} (x_i,t_i) = (x_j(x_i,t_i), t_j(t_i) ) .\]
Les nombres $p$ et $q$ sont respectivement la dimension et la codimension du feuilletage. 

\'Etant donn\'ee une m\'etrique lisse $ds^2$ sur le fibr\'e tangent de $\mathcal F$, on note $\Delta$ le laplacien associ\'e \`a cette m\'etrique. Lucy Garnett a \'etudi\'e dans \cite{Garnett} l'\'equation de la chaleur feuillet\'ee $\frac{\partial u } {\partial t} = \Delta u$, et y d\'eveloppe la th\'eorie ergodique du mouvement brownien le long des feuilles, en d\'emontrant que les mesures stationnaires de ce processus de Markov sont les mesures de probabilit\'e $\mu$ dites \textit{harmoniques}, c'est \`a dire telles que $\Delta \mu = 0$ au sens faible. L'ensemble des mesures harmoniques forme un compact convexe de l'espace des mesures de probabilit\'e sur la vari\'et\'e ambiante. 

Lorsque le feuilletage est lisse et transversalement conforme, et qu'il ne poss\`ede pas de mesure transverse invariante, il est d\'emontr\'e dans~\cite{DK} qu'il n'existe qu'un nombre fini d'ensembles minimaux\footnote{On entend par \textit{minimal} un ensemble ferm\'e $\mathcal F$-satur\'e dans lequel toute les feuilles sont denses.} supportant chacun une unique mesure harmonique, et que de surcro\^{\i}t, toute mesure harmonique est une combinaison convexe de celles-ci. Ainsi, l'ensemble des mesures harmoniques sur un tel feuilletage est un simplexe de dimension $n-1$, o\`u $n$ est le nombre de minimaux de $\mathcal F$.  

Les hypoth\`eses de r\'egularit\'e de ce th\'eor\`eme sont les suivantes~: le feuilletage est de classe $C^1$ transversalement, les feuilles sont des sous-vari\'et\'es immerg\'ees de classe $C^{\infty}$, et la m\'etrique sur ces feuilles varie de fa\c{c}on h\"old\'erienne en fonction du param\`etre transverse. Dans cette note, nous d\'emontrons que si l'on affaiblit l'hypoth\`ese de r\'egularit\'e sur la m\'etrique, le r\'esultat n'est plus valable~: nous construisons un feuilletage lisse par surfaces d'une vari\'et\'e de dimension $3$, dont toutes les feuilles sont denses, et une famille de m\'etriques lisses sur ses feuilles qui d\'ependent contin\^ument du param\`etre transverse (mais pas de fa\c{c}on H\"older), et pour laquelle il existe au moins deux mesures harmoniques diff\'erentes. 

Nous avons donc un exemple de feuilletage o\`u la g\'eom\'etrie du convexe des mesures harmoniques varie en fonction de la m\'etrique choisie sur les feuilles. Des exemples de ce type ont \'et\'e trouv\'e par Victor Kleptsyn et Samuel Petite en toute r\'egularit\'e, mais en codimension sup\'erieure~\cite{Kleptsyn-Petite}. 

Signalons aussi que le feuilletage que nous construisons est minimal mais pas uniquement ergodique. De tels exemples ont \'et\'e construit dans~\cite{Deroin} en toute r\'egularit\'e, en utilisant une construction de type suspension et l'existence d'un diff\'eomorphisme minimal du tore non uniquement ergodique, d\^ue \`a Hillel Furstenberg.

Nous remercions le rapporteur pour ses remarques qui ont am\'elior\'e l'exposition.  

\section{Feuilletage de Hirsch} 

Hirsch a construit un feuilletage lisse de dimension $2$ et de codimension $1$ d'une vari\'et\'e compacte ferm\'ee, associ\'e \`a un diff\'eomorphisme local $T$ du cercle dans lui-m\^eme, de degr\'e topologique $d>1$. Nous rappelons cette construction dans le cas de la transformation $T(z) = z^2$ du cercle unit\'e dans lui-m\^eme. 

Consid\'erons un pantalon orient\'e $P$ et notons ses trois composantes de bord $\partial_i P$, $i=1,2,3$. Soit $\sigma : P\rightarrow P$ une involution lisse qui stabilise la composante $\partial _3 P$ de $\partial P$ et \'echange les composantes $\partial _1 P $ et $\partial _2 P$. Par exemple, on pourra prendre pour $P$ le disque unit\'e de ${\bf C}$ priv\'e des disques de rayon $1/4$ centr\'es en $1/2$ et $-1/2$, et poser $\sigma(x) = -x$. Les composantes $\partial_1 P$, $\partial _2 P$ sont alors respectivement les cercles de rayon $1/4$ de centre $1/2$ et $-1/2$, et $\partial _3 P$ est le cercle unit\'e.

Notons $i(z) = -z$ l'involution du cercle dans lui-m\^eme qui consiste \`a \'echanger les points des fibres de $T$. Le quotient $N = (P \times {\bf S}^1) / (\sigma \times i)$ poss\`ede une structure naturelle de fibration en pantalons $P \rightarrow N \rightarrow {\bf S}^1 / i$, et est diff\'eomorphe \`a un tore solide duquel on a enlev\'e un tore solide int\'erieur qui fait deux fois le tour du premier - voir le survol~\cite{Ghys} dans lequel le lecteur trouvera une jolie figure repr\'esentant ce quotient. Le bord de $N$ est form\'e de deux composantes toriques~: la composante int\'erieure $\partial_i N$ qui  s'identifie naturellement avec $\partial _1 P \times {\bf S}^1 $, et la composante ext\'erieure $\partial_{ext} N$ qui est le quotient de $\partial _ 3 P \times {\bf S}^1$ par le diff\'eomorphisme $\sigma \times i$. On recolle ces deux composantes par le diff\'eomorphisme 
\[  (x, z) \in \partial _3 P \times {\bf S}^1/ \sigma \times i \mapsto ( \frac{x z}{4} + \frac{1}{2} , z^2 )\in \partial _1 P \times {\bf S}^1  .\]
On obtient une vari\'et\'e compacte ferm\'ee $M$. La fibration horizontale par pantalons sur $N$ induit un feuilletage lisse $\mathcal F$ par surfaces~; c'est le feuilletage de Hirsch
 associ\'e \`a $T$. 

Pour construire une m\'etrique sur le fibr\'e tangent du feuilletage de Hirsch, il suffit de construire une famille de m\'etriques $\{ ds^2_ z \}_{z\in {\bf S}^1}$ sur un voisinage ouvert $U$ de $P$ dans ${\bf C}$, telles que 
\begin{itemize}
\item Pour tout $z\in {\bf S}^1$, on a $\sigma^{\star} ds^2 _ z = ds^2 _{i(z)}$.
\item Au voisinage de $\partial _ 3 P$, on a $(\frac{xz}{4} + \frac{1}{2})^{\star} ds^2_{z^2} = ds^2_z$ et $(\frac{xz}{4} - \frac{1}{2})^{\star} ds^2_{z^2} = ds^2_{i(z)}$.  
\end{itemize}  
Une fa\c{c}on simple de construire de telles familles est de consid\'erer des m\'etriques sur $P$ que nous appellerons~\textit{admissibles}. Une m\'etrique admissible est une m\'etrique $ds^2$ sur un voisinage de $P$ dans ${\bf C}$ qui, au voisinage de $\partial _ 3 P$, admet l'expression 
\[ |ds| = \frac{|dz|}{|z| (2\pi + \log \frac{1}{|z|} )},\] 
et v\'erifie 
  de plus $(\frac{x}{4} + \frac{1}{2})^{\star} ds^2 = ds^2$ et $(\frac{x}{4} - \frac{1}{2})^{\star} ds^2 = ds^2$. Alors, pour construire une m\'etrique sur le fibr\'e tangent de $\mathcal F$, il suffit de construire une famille $\{ds^2_z\}_{z\in {\bf S}^1}$ de m\'etriques admissibles qui v\'erifient de surcro\^{\i}t la condition $\sigma ^{\star} ds^2 _ z = ds^2 _ {i(z)}$. En effet, une m\'etrique admissible est invariante par rotation au voisinage de $\partial_ 3 P$, ce qui montre que pour toute paire de m\'etriques admissibles $ds_j ^2$, $j=0,1$, et tout $z$ du cercle, on a $(\frac{xz}{4} \pm \frac{1}{2})^{\star} ds^2_{0} = ds^2_1$.

\section{$g$-mesures et mesures harmoniques}  

On reprend les notations du paragraphe pr\'ec\'edent. Une \textit{$g$-fonction}\footnote{La terminologie est malheureuse mais c'est celle qui est classiquement utilis\'ee.} est une fonction continue $g: {\bf S}^1\rightarrow (1,+\infty)$ telle que pour tout point $z\in {\bf S}^1$,  
\[\frac{1}{g(z)}+ \frac{1}{g(i(z))} =1 .\] 
Une \textit{$g$-mesure} est une mesure de probabilit\'e $\mu$ sur le cercle telle que la d\'eriv\'ee de Radon-Nikodym de $T$ relativement \`a $\mu$ est la fonction $g$. Rappelons que cela signifie que, si $B$ est un Bor\'elien du cercle sur lequel $T$ est injective, alors $\mu ( T B) = \int _B g d\mu$. L'existence d'une $g$-mesure d\'ecoule du th\'eor\`eme du point fixe de Kakutani, voir~\cite{Keane}.

Lucy Garnett d\'emontre dans~\cite{Garnett} qu'il existe toujours une mesure harmonique sur un feuilletage \'equipp\'e d'une m\'etrique sur son fibr\'e tangent, qui est lisse sur les feuilles, et continue transversalement. Dans ce qui suit, nous construisons explicitement des mesures harmoniques dans le cas particulier du feuilletage de Hirsch. Plus pr\'ecis\'ement, \'etant donn\'ee une $g$-fonction associ\'ee \`a $T$, nous produisons une m\'etrique riemannienne sur $T\mathcal F$, qui est lisse le long des feuilles et admet la m\^eme r\'egularit\'e transverse que $g$, en sorte que toute $g$-mesure donne lieu \`a une mesure harmonique sur $\mathcal F$.  

\begin{proposition} Pour tout $\varepsilon >0$, il existe un voisinage $U$ de $P$ dans ${\bf C}$, tel que pour tout couple $(L_1,L_2)$ de r\'eels sup\'erieurs \`a $\varepsilon $ et v\'erifiant 
\[ e^{-L_1} + e^{-L_2} = 1,\]
il existe une m\'etrique riemannienne admissible $ds^2_ {L_1,L_2}$ sur $U$, et une fonction $\Delta_{ds^2_{L_1,L_2}}$-harmonique $\varphi_{L_1,L_2} : U \rightarrow {\bf R} ^{>0}$ qui v\'erifie les conditions suivantes~:
\begin{itemize} 
\item Pour tout $x$ dans un voisinage de $\partial_ 3 P$, on a $\varphi_{L_1,L_2}(x) = 1 + \frac{1}{2\pi}\log \frac{1}{|x|}$. 
\item Pour tout $x$ dans un voisinage de $\partial_3 P$, on a 
\[ \varphi_{L_1, L_2} ( \frac{x}{4} + 1/2 ) = e^{-L_1 }\varphi_{L_1,L_2} (x), \ \ \ \text{et} \ \ \ \varphi_{L_1, L_2} ( \frac{x}{4} - 1/2 ) = e^{-L_2} \varphi_{L_1,L_2} (x) . \]
\end{itemize}
De plus, on peut supposer que les m\'etriques $ds^2_{L_1,L_2}$ et les fonctions $\varphi_{L_1,L_2}$ d\'ependent de fa\c{c}on analytique de $L_1$ et $L_2$, et que pour tout $(L_1,L_2)$, on a $\sigma^{\star} ds^2_{L_1,L_2} = ds^2_{L_2,L_1}$.\end{proposition}  

\begin{proof} 
Sur les cylindres $C_i={\bf S}^1 \times [0,L_i]$, $i=1,2$, consid\'erons la m\'etrique de courbure $-1$ d\'efinie par~: 
\[ e^{2(v-L_i )} du^2 + dv^2.\] 
En effet, c'est la m\'etrique qu'on obtient en partant de $\frac{du^2 + dy^2}{y^2}$ et en effectuant le changement de variables $y = e^{L_i-v}$. Les bords $\partial_- C_i = {\bf S}^1 \times 0$ et $\partial _+ C_i = {\bf S}^1 \times L_i$ sont alors des \textit{horocycles} respectivement n\'egatif de longueur $e^{-L_i}$ et positif de longueur $1$\footnote{Nous entendons par horocycle positif ou n\'egatif une courbe lisse de courbure sign\'ee $1$ ou $-1$.}.

On coupe $C_1$ et $C_2$ le long des g\'eod\'esiques $1\times [0,\varepsilon]$ et  $1\times [0,\varepsilon]$, et on colle le segment $1_+ \times [0,\varepsilon]$ de $C_1$ (resp.  $1_- \times [0,\varepsilon]$ de $C_2$) au segment $1_- \times [0,\varepsilon]$ de $C_2$ (resp.  $1_+ \times [0,\varepsilon]$ de $C_1$) de fa\c{c}on isom\'etrique et en renversant l'orientation. On construit de cette fa\c{c}on un pantalon $P_{L_1,L_2}$ avec une m\'etrique $ds^2$ de courbure $-1$ et une singularit\'e conique d'angle $4\pi$. Ce pantalon $P_{L_1,L_2}$ a trois composantes de bord~: les composantes $\partial _{1} P_{L_1,L_2} = \partial _+ C_1$, $\partial _2 P_{L_1,L_2} = \partial _+ C_{2}$, qui sont des horocycles positifs de longueur $1$, et la composante $\partial _3 P_{L_1,L_2}$ qui est un horocycle n\'egatif de longueur la somme des longueurs des bords $\partial _- C_{1}$ et $\partial _- C_{2}$, c'est \`a dire $e^{-L_1} + e^{-L_2}= 1$. 

La m\'etrique avec singularit\'e conique munit $P_{L_1,L_ 2}$ d'une structure de surface de Riemann lisse - l'atlas des cartes pr\'eservant l'orientation dans lesquelles la m\'etrique est conforme \`a la m\'etrique plate $|dz|$ sur ${\bf C}$. La fonction $\varphi_{L_1,L_2}: P_{L_1,L_2} \rightarrow {\bf R}$ d\'efinie sur chaque $C_i$ par $e^{-v}$ est alors une fonction \textit{harmonique} sur $P_{L_1,L_2}$, qui vaut $e^{-L_i}$ sur $\partial _ i P_{L_1,L_2}$ pour $i=1,2$, et $1$ sur $\partial _ 3 P_{L_1,L_2}$.

Les bords de $P_{L_1,L_2}$ \'etant horocycliques de longueur $1$, il existe un diff\'eomorphisme $\Phi : P \rightarrow P_{L_1,L_2}$ tel que $\Phi ^{\star} ds^2$ est une m\'etrique admissible, \`a ceci pr\`es qu'elle admet une singularit\'e conique. On peut choisir $\Phi$ en sorte que cette derni\`ere se situe \`a l'origine, et que l'on ait en son voisinage $\Phi^{\star} ds^2 = |x| ^2 |dx|^2$. On consid\`ere alors une m\'etrique de la forme $ds^2 _{L_1,L_2} = \rho \Phi^{\star} ds^2$, o\`u $\rho : P \setminus \{0\}\rightarrow {\bf R}^{>0}$ est une fonction lisse, qui vaut identiquement $1$ \`a l'ext\'erieur d'un petit voisinage de l'origine, et qui, dans un voisinage encore plus petit, est de la forme $\rho (x) = \frac{1}{|x|^2}$.

La fonction $\varphi_{L_1,L_2}\circ \Phi$ est alors $\Delta_{ds^2_{L_1,L_2}}$-harmonique, et v\'erifie les conditions du lemme. 
\end{proof} 

Nous choisissons $\varepsilon $ de sorte que $0< \varepsilon < \inf _{z\in {\bf S}^1} \log g (z)$. Pour chaque point $z$ du cercle, on pose 
\[  L_1(z) = \log g(z),\ \ \ \ \text{et} \ \ \ \ L_2(z) = \log g ( i(z) ) .\]
La famille de m\'etriques admissibles $\{  ds_z ^2 \} _{z \in {\bf S}^1} $ d\'efinies par $ds^2 _ z = ds^ 2 _{L_1(z), L_2(z)}$ d\'efinit alors une m\'etrique sur le feuilletage de Hirsch. D'autre part, si $\mu$ est une $g$-mesure sur le cercle, alors la mesure 
\[  m = \varphi_{L_1(z), L_2(z)} \text{vol} (ds^2 _z ) \otimes \mu \]
d\'efinit une mesure harmonique sur le feuilletage de Hirsch. Pour conclure, il nous suffit de prendre une $g$-fonction continue pour laquelle il existe plusieurs $g$-mesures diff\'erentes, dont l'existence nous est assur\'ee par un th\'eor\`eme de Anthony N. Quas~\cite{Quas}.

\vspace{0.1cm}

\vspace{0.3cm}

\end{small}

\end{document}